\documentclass{amsart}
\newtheorem{thm}{Theorem}
\newtheorem{prop}{Proposition}
\newtheorem{question}{Question}
\newcommand{\p}{^{\prime}}
\newcommand{\pp}{^{\prime\prime}}

\title{Compositions that are palindromic modulo $m$}
\author{Matthew Just}
\date{}

\begin{document}

\maketitle

\begin{abstract}
    In recent work, G. E. Andrews and G. Simay prove a surprising relation involving parity palindromic compositions, and ask whether a combinatorial proof can be found. We extend their results to a more general class of compositions that are palindromic modulo $m$, that includes the parity palindromic case when $m=2$. We provide combinatorial proofs for the cases $m=2$ and $m=3$.
\end{abstract}

\section{Introduction}

Let $\sigma=(\sigma_1,\sigma_2,\ldots,\sigma_k)$ be a sequence of positive integers such that $\sum \sigma_i = n$. The sequence $\sigma$ is called a \textit{composition} of $n$ of length $k$. The numbers $\sigma_i$ are called the \textit{parts} of the composition. If $\sigma_i=\sigma_{k-i+1}$ for all $i$, then $\sigma$ is called a \textit{palindromic} composition. If instead $\sigma$ satisfies the weaker condition that $\sigma_i \equiv \sigma_{k-i+1}$ modulo $m$ for some $m\geq 1$ and all $i$, then $\sigma$ is said to be \textit{palindromic modulo $m$}. Let $pc(n,m)$ be the number of compositions of $n$ that are palindromic modulo $m$. 

Andrews and Simay \cite{AS20} have shown that $pc(1,2)=1$, and \[pc(2n,2)=pc(2n+1,2)=2\cdot 3^{n-1}\]
for $n>1$.
Our main theorem generalizes their result by giving an ordinary generating function for $pc(n,m)$ for all $m$. Of particular interest is that $pc(1,3)=1$, and \[pc(n,3)=2\cdot f(n-1)\] for $n>1$, where $f(n)$ is the $n$th Fibonacci number (here $f(1)=f(2)=1$). 
We then give a combinatorial proof of the above identities for $pc(n,2)$ and $pc(n,3)$, and conclude with some general properties and asymptotic analysis of $pc(n,m)$ for larger $m$.

\begin{thm}
    We have \[F_m(q) := \sum_{n\geq 1} pc(n,m) q^n = \frac{q+2q^2-q^{m+1}}{1-2q^2-q^m}. \] 
\end{thm}

\noindent{\bf Remark.} Though we mean this to be a formal definition, it can be verified that the series converges for all $|q|<1/2$. Indeed, \[F_m(q) \leq F_1(q) = \frac{q}{1-2q}.\]

\begin{proof}[Proof of Theorem 1]
    Notice that to form a composition that is palindromic modulo $m$, we need to form a sequence of pairs of positive integers whose difference is a multiple of $m$. Furthermore, there may or may not be a central part that will always be congruent to itself. If $G_m(q)$ is the ordinary generating function for pairs of positive integers whose difference is a multiple of $m$, then \[F_m(q) = \frac{G_m(q)}{1-G_m(q)} \cdot \frac{1}{1-q} + \frac{q}{1-q}.\] Now for the function $G_m(q)$, either the pair of positive integers is the same, or we must choose to add a multiple of $m$ to either the first or second number. Therefore \[G_m(q) = \frac{q^2}{1-q^2}+\frac{2q^2}{1-q^2}\cdot \frac{q^m}{1-q^m}.\] Simplifying gives the result. \end{proof}
    
When $m=1$, every composition will be palindromic modulo $m$. Therefore $pc(n,1)=2^{n-1}$, as is evident from $F_1(q)$. We next consider the case when $m=2$.

\section{The case $m=2$}

When $m=2$, Theorem 1 gives \[F_2(q) = \frac{2q^3(1+q)}{1-3q^2},\] and after expanding we see that $p(1,2)=1$, and $p(2n,2)=p(2n+1,2)=2\cdot 3^{n-1}$. Our goal in this section will be to give a combinatorial proof of this formula. 

For $n\geq 1$, we start by considering the $3^{n-1}$ elements of the set $\{0,1,2 \}^{n-1}$. Our goal will be to embed two disjoint copies of this set into the set of compositions of $2n$ that are palindromic modulo 2, and show that every composition of $2n$ that is palindromic modulo 2 has a preimage. We then give a bijection between compositions of $2n$ that are palindromic modulo 2 and compositions of $2n+1$ that are palindromic modulo 2.

\begin{proof}[Combinatorial proof that $pc(2n,2)=2\cdot3^{n-1}$] For $n\geq 1$, let \[a=(a_1,a_2,\ldots,a_{n-1})\in \{0,1,2 \}^{n-1}.\] As an example, we take $a=(0,1,2,2,1)\in \{0,1,2 \}^5$. We will first use the sequence of $a_i$s to construct a sequence of triples $(b_j,c_j,d_j)$, where each triple will become a pair of  parts $\{ \sigma_j,\sigma_{k-j+1}\}$ or the central part $\sigma_{(k+1)/2}$ in the resulting composition having length $k$.

Initialize the process with the triple $(b_1,c_1,d_1)=(1,0,0)$. If $n=1$, then we are done. Now if $a_1=0$, we create a new triple $(b_2,c_2,d_2)=(1,0,0)$. If $a_1=1$, then we increase $b_1$ by one. If $a_1=2$, then we increase $c_1$ by one. In our example of $(0,1,2,2,1)$, we see that $a_1=0$, so after the first step we have the following list of triples
\[(1,0,0),  \ \ (1,0,0).\]

Now for $i>1$, we define the instructions given by each $a_i$ recursively. We assume that at this step we have a total of $j$ triples, for some $j\geq 1$. 
\begin{enumerate}
    \item If $a_i=0$, then we create the new triple $(b_{j+1},c_{j+1},d_{j+1})=(1,0,0)$.
    \item If $a_i=2$, then increase $c_j$ by one.
    \item If $a_i=1$ and $a_{i-1}\neq 2$, then increase $b_j$ by one. 
    \item If $a_i=1$ and $a_{i-1}=2$, then set $d_j=1$ and create the new triple $(b_{j+1},c_{j+1},d_{j+1})=(1,0,0)$.
\end{enumerate}
   For our example $(0,1,2,2,1)$, we form the list of triples
\[(1,0,0), \ \ (2,2,1), \ \  (1,0,0). \]

After performing this algorithm on each of the $3^{n-1}$ sequences in $\{ 0,1,2 \}^{n-1}$ gives a set of $3^{n-1}$ sequences of triples. For each sequence $a\in\{0,1,2\}^{n-1}$, we denote by $k_a$ the number triples in the list we have formed. By construction, all of these sequences of $k_a$ triples are distinct.

We create a second set containing $3^{n-1}$ sequences of triples by setting $d_{k_a} = 1$ for all $a$. For our example $(0,1,2,2,1)$, we get a second sequence of triples
\[(1,0,0), \ \ (2,2,1), \ \  (1,0,1). \]
We now have $2\cdot 3^{n-1}$ sequences of triples, all of which are distinct.

Next, for each sequence $a$, taking one of its two associated sequence of triples $(b_1,c_1,d_1)$, \ldots, $(b_{k_a},c_{k_a},d_{k_a})$, we show how to form a unique composition $\sigma$ of $2n$ that is palindromic modulo 2. Start with the triple $(b_{k_a}, c_{k_a}, d_{k_a})$. 
\begin{enumerate}
    \item If $d_{k_a}=0$ and $c_{k_a}=0$, then set $\sigma_{k_a}=2b_{k_a}$.
    \item If $d_{k_a}=0$ and $c_{k_a}>0$, then set $\sigma_{k_a}=b_{k_a}+2c_{k_a}$ and $\sigma_{k_a+1}=b_{k_a}$.
    \item If $d_{k_a}=1$, then set $\sigma_{k_a}=b_{k_a}$ and $\sigma_{k_a+1}=b_{k_a}+2c_{k_a}$.
\end{enumerate}
  These cases will determine whether the length of the composition is even or odd. If $d_{k_a}=0$ and $c_{k_a}=0$, then the length of $\sigma$ will be $2k_a-1$. Otherwise, the length of $\sigma$ will be $2k_a$. For convenience we refer to the length of $\sigma$ as $k$ in either case. For our example $(0,1,2,2,1)$, we have initialized the two compositions \[(0,0,2,0,0) \ \text{and} \ (0,0,1,1,0,0).\]

Now if $k_a=1$ we are done. Assuming $k_a>1$, for each $1\leq j < k_a$ we consider the triple $(b_j,c_j,d_j)$. If $d_j=0$, then set $\sigma_j=b_j+2c_j$ and $\sigma_{k-j+1} = b_j$. If $d_j=1$, then set $\sigma_j=b_j$ and $\sigma_{k-j+1}=b_j+2c_j$. Also notice if $d_j=1$, then by construction $c_j>0$, so that $b_j\neq b_j+2c_j$. This ensures that each sequence of triples will be associated to a unique composition. Also notice that the sum of the parts in $\sigma$ equals $2\sum b_j + 2\sum c_j=2n$, and that $|\sigma_j-\sigma_{k-j+1}|$ is even for all $j$. Thus $\sigma$ is a composition of $2n$ that is palindromic modulo 2. Showing that every composition of $2n$ that is palindromic modulo 2 can be formed from one of the described sequences of triples is done by reversing the described algorithm. Returning to our example sequence $(0,1,2,2,1)$, we have form the two compositions
\[(1,2,2,6,1) \ \text{and} \ (1,2,1,1,6,1),\]
each of which is a composition of 12 that is palindromic modulo 2.\end{proof}

\begin{proof}[Combinatorial proof that $pc(2n,2)=pc(2n+1,2)$]
    We split the proof into two cases. Starting with a composition $\sigma$ of $2n$ that is palindromic modulo 2, suppose the length of $\sigma$ is $2k+1$. Then adding one to $\sigma_{k+1}$ gives a composition of $2n+1$ that is still palindromic modulo 2. Now if the length of $\sigma$ is $2k$, form a composition $\sigma^{\prime}$ of $2n+1$ of length $2k+1$ by setting 
    \[\sigma^{\prime}_j=\begin{cases}
        \sigma_j & 1\leq j\leq k \\
        1 & j=k+1 \\
        \sigma_{j-1} & k+2\leq j \leq 2k+1.
    \end{cases}
    \]
We note that $\sigma^{\prime}$ is still palindromic modulo 2. It is straightforward to verify that this map is a bijection. For our examples of $(1,2,2,6,1)$ and $(1,2,1,1,6,1)$ that are compositions of 12, we form the two compositions of 13 
\[(1,2,1,1,1,6,1) \ \ \text{ and } \ \ (1,2,3,6,1),\] 
each of which is palindromic modulo 2. \end{proof}

\section{The case $m=3$}

When $m=3$, Theorem 1 gives 
\[F_3(q) = \frac{q+2q^2-q^4}{1-2q^2-q^3} = q + \frac{2q^2}{1-q-q^2},\]
and expanding this function shows $pc(1,3)=1$ and $pc(n,3)=2\cdot f(n-1)$ for $n>1$.  Our goal of this section is to give a combinatorial proof of this formula.

It is well known \cite{AH75} that $f(n+1)$ is equal to the number of compositions of $n$ with parts that are equal to one or two. For $n>1$, our goal will be to embed two disjoint copies of the compositions of $n-2$ with parts equal to one or two into the compositions of $n$ that are palindromic modulo 3. Then we will show each composition of $n$ that is palindromic modulo 3 has a preimage. 

\begin{proof}[Combinatorial proof that $pc(n,3)=2\cdot f(n-1)$.] For $n>1$, let $\sigma =( \sigma_1,\sigma_2,\ldots,\sigma_k)$ be a composition of $n-2$, where each $\sigma_i\in \{1,2\}$. Form two distinct compositions of $n$ by setting 
\begin{align*}
    \sigma\p&=(1,1,\sigma_1,\sigma_2,\ldots,\sigma_k) \text{, and} \\
    \sigma\pp&=(2,\sigma_1,\sigma_2,\ldots,\sigma_k).
\end{align*}
Doing this for each composition gives us a set of $2\cdot f(n-1)$ compositions of $n$ with parts equal to one or two, the set of which we will denote by $A_n$. For each of these compositions, we will form a unique composition of $n$ that is palindromic modulo 3.

Let $+$ denote sequence concatenation, as in $(1,2)+(3,4)=(1,2,3,4)$. Now for any composition $a\in A_n$, we can decompose $a$ in the following way:
\[a=a_1+a_2+\ldots+a_s,\]
where the substrings $\{a_i\}_{i=1}^s$ are determined by the following rules. We will write $k_i$ to be the length of each substring $a_i$.

\begin{enumerate}
    \item If $a$ contains no twos, then $a_1=a$.
    \item Suppose $a$ begins with a two, and this is the only two in $a$. Then $a_1=(2)$.
    \item Suppose $a$ begins with the substring $(2,2)$ or $(2,1,1)$. Then $a_1=(2)$.
    \item Suppose $a$ begins with the substring $(1,1)$ or $(2,1,2)$. Then $a_1$ terminates with the first two such that
    \begin{enumerate}
        \item it is the final two in $a$;
        \item it is immediately followed by the substring $(2)$ or $(1,1)$.
    \end{enumerate}
   \end{enumerate}
For $i>1$, the substring $a_i$ is determined recursively using these 4 rules, after first deleting the substrings $a_1$, $a_2$, \ldots, $a_{i-1}$ from $a$. The resulting string cannot begin with the substring $(1,2)$, so it still makes sense to apply these rules. To aid the reader, Table \ref{decomp} shows the decomposition of the compositions in the set $A_8$. As an example, take the composition $a=(1,1,1,2,2,1,1,2,1,2,1,1)$ in $A_{16}$. By applying rule (4) we have $a_1=(1,1,1,2)$, and we now consider the string $(2,1,1,2,1,2,1,1)$. By applying rule (3) we have $a_2=(2)$, and we now consider the string $(1,1,2,1,2,1,1)$. By applying rule (4) we have $a_3=(1,1,2,1,2)$, and we now consider the string $(1,1)$. By applying rule (1) we have $a_4=(1,1)$, and are now done. Thus
\[(1,1,1,2,2,1,1,2,1,2,1,1)=(1,1,1,2)+(2)+(1,1,2,1,2)+(1,1).\]

Let $B_n$ be the set of compositions of $n$ that are palindromic modulo 3. For each composition $a=a_1+\ldots+a_s\in A_n$, we will show how to construct a unique $b\in B_n$. If $a_s$ contains a two, the length of $b$ will be $2s$; if $a_s$ does not contain a two, the length of $b$ will be $2s-1$. Either way, we will denote $k_b$ to be the length of $b$ and write $b=(b_1,b_2,\ldots,b_{k_b})$ (recall we have also set $k_i$ to be the length of $a_i$).

Now assume $a_i$ has $o_i$ ones, and $t_i$ twos. If $t_i=0$ (which can only be the case for $a_s$), then set $b_s=k_s$. For what follows we will assume $t_i>0$. Note that $k_i=o_i+t_i$, and \[n=\sum_{i=1}^s (o_i+2\cdot t_i).\]
We will now form the triple $(c_i,d_i,e_i)$ using the following rules.

Let $o\p_i$ be the number of ones in $a_i$ preceding the first two. Then if
\begin{enumerate}
    \item $o\p_i$ is even (possibly zero), set 
    \[c_i=\frac{o\p_i}{2}+1, \ \ d_i=3\cdot(t_i-1), \ \ \text{and} \ \ e_i=0. \]
    \item $o\p_i$ is odd, set 
    \[c_i=\frac{o\p_i-1}{2}, \ \ d_i=0, \ \ \text{and} \ \ e_i=3\cdot t_i. \]
\end{enumerate}

By construction this sequence of triples is unique to the composition $a\in A_n$ we began with. We show how our example composition with decomposition $a_1=(1,1,1,2)$, $a_2=(2)$, $a_3=(1,1,2,1,2)$, and $a_4=(1,1)$ maps to three triples (the final substring $a_4$ contains no twos, and therefore we set $b_4=k_4=2$).
\begin{center}
    \begin{tabular}{c||c|c|c|c|c|c}
         $a_i$ & $o_i$ & $o\p_i$ & $t_i$ & $c_i$ & $d_i$ & $e_i$  \\
         \hline
         $(1,1,1,2)$    &3&3&1&1&0&3\\
         $(2)$          &0&0&1&1&0&0\\
         $(1,1,2,1,2)$  &3&2&2&2&3&0
    \end{tabular}
\end{center} 
Therefore, we have the triples
\[(1,0,3), \ \ (1,0,0), \ \ \text{ and } \ \ (2,3,0).\]

We now form $b$ by setting $b_i=c_i+d_i$ and $b_{k_b-i+1}=c_i+e_i$, adjusting slightly for the case when $a_s$ has no twos (recall that we just set $b_s=k_s$ in this case). The composition is uniquely determined from the triples, and by construction $d_i$ and $e_i$ are both multiples of 3. Furthermore, 
\[b_i+b_{k_b-i+1} = 2c_i + d_i+e_i = o\p_i +3t_i -1=o_i+2t_i. \]
To see why the last equality holds, consider first the case when $t_i=1$. Then $o\p_i=o_i$, and $3t_i-1=2t_i$. Now if $t_i>1$, we can pair each two in $a_i$ with a one immediately preceding it. However, we have used a one counted by $o\p_1$ in this pairing, so we must subtract one. Therefore, we have embedded $A_n$ into $B_n$. For our example
\[a=(1,1,1,2,2,1,1,2,1,2,1,1)=(1,1,1,2)+(2)+(1,1,2,1,2)+(1,1),\]
we have the triples $(1,0,3)$, $(1,0,0)$, and $(2,3,0)$, and we obtain the composition
\[b=(1,1,5,2,2,1,4),\]
which is a composition of 16 that is palindromic modulo 3 ($b\in B_{16}$). It is straightforward to construct a composition $a\in A_n$ from a composition $b\in B_n$ by reversing this construction, which proves the result. \end{proof}

\section{The case $m>3$}

If $m>3$, the formula for $pc(n,m)$ can be deduced from Theorem 1; we only give a detailed study for $m=2$ and $m=3$ based on the elegance of the formulae. For $m=4$, Theorem 1 gives 
\[F_4(q)=\frac{q+2q^2-q^5}{1-2q^2-q^4},\]
and after expanding we see $pc(1,4)=1$, and for $n\geq 1$
\[pc(2n,4)=pc(2n+1,4)=\frac{(1+\sqrt{2})^n-(1-\sqrt{2})^n}{\sqrt{2}}.\]
The familiar reader will recognize these numbers as twice the Pell numbers \cite{pell}. While the correct context for a combinatorial proof is not immediately clear, we pose the following question.

\begin{question}
    Is there a combinatorial proof of the formula for $pc(n,4)$?
\end{question}

We also include some properties of $pc(n,m)$ that we observed for general $m$.

\begin{prop}
    If $m$ is even, then for $n\geq 1$ we have $pc(2n,m)=pc(2n+1,m)$.
\end{prop}

\begin{proof}
    The proof of this is identical in spirit to the combinatorial proof that $pc(2n,2)=pc(2n+1,2)$ in Section 2. \end{proof}

\begin{prop}
    If $n>1$, then $pc(n,m)$ is even for all $m$.
\end{prop}

\begin{proof}
    A quick, one line proof of this fact comes from the identity
    \[F_m(q)-q = \frac{2q^2 (1+q)}{1-2q^2-q^m}.\]
    This can also be seen combinatorially by pairing up compositions of $n$ that are palindromic modulo $2$. Let $a$ be a composition of $n>1$ of length $s$ that is palindromic modulo $m$. If $a_i\neq a_{s-i+1}$ for some $i$, then we can pair this composition with the composition formed by switching $a_i$ and $a_{s-i+1}$. 
    
    Now assume $a_i=a_{s-i+1}$ for all $i$ (i.e. $a$ is a palindrome). If $s$ is odd, than $a_{(s+1)/2}$ is even, and we can pair this composition with the composition of $n$ of length $s+1$ formed by removing $a_{(s+1)/2}$ and appending $a_{(s+1)/2}/2$ to the beginning and the end. \end{proof}

\begin{prop}
    If $2n\geq m$, then $p(2n,m)=p(2n+1,m)=2^n$.
\end{prop}

\begin{proof}
    The condition that $2n\geq m$ requires all compositions of $2n$ and $2n+1$ that are palindromic modulo $m$ to be palindromes. Let $pc(n,\infty)$ denote the number of palindromic compositions of $n$. Then it is well known \cite{wk} that
    \[pc(n,\infty)=2^{\lfloor n/2\rfloor}. \qedhere\]
\end{proof}

\begin{prop}
    Let $\alpha_m$ be the unique positive root of $1-2q^2-q^m=0$, and set 
    \[c_m=\lim_{q\rightarrow \alpha_m} (1-\alpha_m^{-1}q)\cdot F_m(q) \ \ \text{ and } \ \ d_m=\lim_{q\rightarrow \alpha_m} (1+\alpha_m^{-1}q)\cdot F_m(q).\]
   If $m$ is even, then
    \[\limsup_{n\rightarrow \infty} \alpha_m^n\cdot pc(n,m) = c_m+d_m \ \ \text{ and } \ \ \liminf_{n\rightarrow\infty} \alpha_m^n\cdot pc(n,m)= c_m-d_m. \]
    If $m$ is odd, $d_m=0$, and thus $pc(n,m)\sim c_m \alpha_m^{-n}$ as $n\rightarrow \infty$.
\end{prop}

\begin{proof}
    This is a routine analysis of the rational function $F_m(q)$. For convenience, we set $\phi_m(q)=1-2q^2-q^m$. The asymptotics of $pc(n,m)$ are determined by the poles of $F_m(q)$, which are precisely the zeros of $\phi_m(q)$. We will assume $m>4$, as the result can be verified directly for $m\leq 4$. 
    
    First note that $\phi_m(q)$ has exactly one positive zero, $\alpha_m$, since $\phi_m\p(q)<0$ for all $q>0$, $\phi_m(0)>0$, and $\phi_m(1)<0$. When $m$ is even, $\phi_m(q)$ is an even function and thus $-\alpha_m$ is also a zero. When $m$ is odd, $\phi_m(q)$ has a unique zero in the interval $(-1,-\alpha_m)$, since $\phi\p_m(q)>0$ on $(-1,0)$, $\phi_m(-1)<0$, and $\phi_m(\alpha_m)>0$. Our next goal is to show that these are the only two zeros inside the disc $|q|<0.9$.
    
    Note that $q^2(2+q^{m-2})$ has one zero (with multiplicity equal to 2) inside the disc $|q|<0.9$. Therefore, if we can show
    \[|q^2(2+q^{m-2}|>1\]
    for all $q$ with $|q|=0.9$, we can apply Rouche's theorem to conclude that $\phi_m(q)$ has exactly two zeros inside the disc $|q|<0.9$. For $m>4$ we have $|q^{m-2}|\leq |q^3|$, so that $|2+q^{m-2}|\geq 2-|q^3|=1.279$ and
    \[|q^2(2+q^{m-2})| > 0.81\cdot 1.729 = 1.03599>1.\]
    
    We can now perform a partial fraction decomposition $F_m(q)$,
    \[F_m(q) = \frac{c_m}{1-\alpha_m^{-1}q} + \frac{d_m}{1+\alpha_m^{-1}q} +G(q),\]
    where $G(q)$ is a rational function with no poles in the disc $|q|<0.9$. Therefore,
    \[pc(n,m) = (c_m + (-1)^n d_m)\alpha_m^{-n} + o(\alpha_m^{-n}) \]
    as $n\rightarrow \infty$, noting that $d_m=0$ if $m$ is odd. The result follows. \end{proof}
    
    The following table shows the values of $\alpha_m$, $c_m$, and $d_m$ for various $m$.
    \begin{center}
        \begin{tabular}{c||c|c|c}
             $m$&$\alpha_m^{-1}$&$c_m$&$d_m$  \\
             \hline
             1& $2$&$\frac{1}{2}=0.5$& 0 \\
             2& $\sqrt{3}\approx 1.73$&$\frac{3+\sqrt{3}}{9}\approx0.53$& $\frac{3-\sqrt{3}}{9}\approx 0.14$ \\
             3& $\frac{\sqrt{5}+1}{2}\approx 1.61$&$\frac{5-\sqrt{5}}{5}=0.55$& 0 \\
             4& $\approx 1.55$&$\approx0.58$& $\approx 0.13$ \\
             \vdots&\vdots & \vdots & \vdots \\
             $\infty$ &$\sqrt{2}\approx 1.41$&$\frac{2+\sqrt{2}}{4}\approx 0.85$& $\frac{2-\sqrt{2}}{4}\approx 0.14$
        \end{tabular}
    \end{center}

\section*{Acknowledgements}

The author was partially supported by the Research and Training Group grant DMS-1344994 funded by the National Science Foundation.

{\footnotesize
\begin{table}[]
    \centering
    \begin{tabular}{l||l||l||l}
         $a\in \{0,1,2 \}^{3}$ & Sequence of triples & Composition of $8$  & Composition of $9$  \\
         \hline
         \hline
         (0,0,0)&(1,0,0),(1,0,0),(1,0,0),(1,0,0)&(1,1,1,2,1,1,1)&(1,1,1,3,1,1,1)\\
                &(1,0,0),(1,0,0),(1,0,0),(1,0,1)&(1,1,1,1,1,1,1,1)&(1,1,1,1,1,1,1,1,1)\\
         (0,0,1)&(1,0,0),(1,0,0),(2,0,0)&(1,1,4,1,1)&(1,1,5,1,1)\\
                &(1,0,0),(1,0,0),(2,0,1)&(1,1,2,2,1,1)&(1,1,2,1,2,1,1)\\
         (0,0,2)&(1,0,0),(1,0,0),(1,1,0)&(1,1,3,1,1,1)&(1,1,3,1,1,1,1)\\
                &(1,0,0),(1,0,0),(1,1,1)&(1,1,1,3,1,1)&(1,1,1,1,3,1,1)\\
         (0,1,0)&(1,0,0),(2,0,0),(1,0,0)&(1,2,2,2,1)&(1,2,3,2,1)\\
                &(1,0,0),(2,0,0),(1,0,1)&(1,2,1,1,2,1)&(1,2,1,1,1,2,1)\\
         (0,1,1)&(1,0,0),(3,0,0)&(1,6,1)&(1,7,1)\\
                &(1,0,0),(3,0,1)&(1,3,3,1)&(1,3,1,3,1)\\
         (0,1,2)&(1,0,0),(2,1,0)&(1,4,2,1)&(1,4,1,2,1)\\
                &(1,0,0),(2,1,1)&(1,2,4,1)&(1,2,1,4,1)\\
         (0,2,0)&(1,0,0),(1,1,0),(1,0,0)&(1,3,2,1,1)&(1,3,3,1,1)\\
                &(1,0,0),(1,1,0),(1,0,1)&(1,3,1,1,1,1)&(1,3,1,1,1,1,1)\\
         (0,2,1)&(1,0,0),(1,1,1),(1,0,0)&(1,1,2,3,1)&(1,1,3,3,1)\\
                &(1,0,0),(1,1,1),(1,0,1)&(1,1,1,1,3,1)&(1,1,1,1,1,3,1)\\
         (0,2,2)&(1,0,0),(1,2,0)&(1,5,1,1)&(1,5,1,1,1)\\
                &(1,0,0),(1,2,1)&(1,1,5,1)&(1,1,1,5,1)\\
         (1,0,0)&(2,0,0),(1,0,0),(1,0,0)&(2,1,2,1,2)&(2,1,3,1,2)\\
                &(2,0,0),(1,0,0),(1,0,1)&(2,1,1,1,1,2)&(2,1,1,1,1,1,2)\\
         (1,0,1)&(2,0,0),(2,0,0)&(2,4,2)&(2,5,2)\\
                &(2,0,0),(2,0,1)&(2,2,2,2)&(2,2,1,2,2)\\
         (1,0,2)&(2,0,0),(1,1,0)&(2,3,1,2)&(2,3,1,1,2)\\
                &(2,0,0),(1,1,1,)&(2,1,3,2)&(2,1,1,3,2)\\
         (1,1,0)&(3,0,0),(1,0,0)&(3,2,3)&(3,3,3)\\
                &(3,0,0),(1,0,1)&(3,1,1,3)&(3,1,1,1,3)\\
         (1,1,1)&(4,0,0)&(8)&(9)\\
                &(4,0,1)&(4,4)&(4,1,4)\\
         (1,1,2)&(3,1,0)&(5,3)&(5,1,3)\\
                &(3,1,1)&(3,5)&(3,1,5)\\
         (1,2,0)&(2,1,0),(1,0,0)&(4,2,2)&(4,3,2)\\
                &(2,1,0),(1,0,1)&(4,1,1,2)&(4,1,1,1,2)\\
         (1,2,1)&(2,1,1),(1,0,0)&(2,2,4)&(2,3,4)\\
                &(2,1,1),(1,0,1)&(2,1,1,4)&(2,1,1,1,4)\\
         (1,2,2)&(2,2,0)&(6,2)&(6,1,2)\\
                &(2,2,1)&(2,6)&(2,1,6)\\
         (2,0,0)&(1,1,0),(1,0,0),(1,0,0)&(3,1,2,1,1)&(3,1,3,1,1)\\
                &(1,1,0),(1,0,0),(1,0,1)&(3,1,1,1,1,1)&(3,1,1,1,1,1,1)\\
         (2,0,1)&(1,1,0),(2,0,0)&(3,4,1)&(3,5,1)\\
                &(1,1,0),(2,0,1)&(3,2,2,1)&(3,2,1,2,1)\\
         (2,0,2)&(1,1,0),(1,1,0)&(3,3,1,1)&(3,3,1,1,1)\\
                &(1,1,0),(1,1,1)&(3,1,3,1)&(3,1,1,3,1)\\
         (2,1,0)&(1,1,1),(1,0,0),(1,0,0)&(1,1,2,1,3)&(1,1,3,1,3)\\
                &(1,1,1),(1,0,0),(1,0,1)&(1,1,1,1,1,3)&(1,1,1,1,1,1,3)\\
         (2,1,1)&(1,1,1),(2,0,0)&(1,4,3)&(1,5,3)\\
                &(1,1,1),(2,0,1)&(1,2,2,3)&(1,2,1,2,3)\\
         (2,1,2)&(1,1,1),(1,1,0)&(1,3,1,3)&(1,3,1,1,3)\\
                &(1,1,1),(1,1,1)&(1,1,3,3)&(1,1,1,3,3)\\
         (2,2,0)&(1,2,0),(1,0,0)&(5,2,1)&(5,3,1)\\
                &(1,2,0),(1,0,1)&(5,1,1,1)&(5,1,1,1,1)\\
         (2,2,1)&(1,2,1),(1,0,0)&(1,2,5)&(1,3,5)\\
                &(1,2,1),(1,0,1)&(1,1,1,5)&(1,1,1,1,5)\\
         (2,2,2)&(1,3,0)&(7,1)&(7,1,1)\\
                &(1,3,1)&(1,7)&(1,1,7)
    \end{tabular}
    \caption{The sequences in $\{0,1,2 \}^{n-1}$ ($n=4$), mapped to two sequences of triples, mapped to a composition of $2n=8$ and a composition of $2n+1=9$, each of which is palindromic modulo 2.}
    \label{decomp1}
\end{table} }

{\scriptsize
\begin{table}[]
    \centering
    \begin{tabular}{l||l||l}
        $a\in A_8$& $a=a_1+a_2+\ldots+a_s$& $b\in B_8$  \\
        \hline
        \hline
        $(1,1,1,1,1,1,1,1)$ & $a_1=(1,1,1,1,1,1,1,1)$& $(8)$\\
        \hline 
        $(1,1,2,1,1,1,1)$ & $a_1=(1,1,2)$&$(2,4,2)$\\
                          & $a_2=(1,1,1,1)$&\\
        \hline
        $(1,1,1,2,1,1,1)$ & $a_1=(1,1,1,2)$&$(1,3,4)$\\
                          & $a_2=(1,1,1)$&\\
        \hline
        $(1,1,1,1,2,1,1)$ & $a_1=(1,1,1,1,2)$&$(3,2,3)$\\
                          & $a_2=(1,1)$&\\
        \hline
        $(1,1,1,1,1,2,1)$ & $a_1=(1,1,1,1,1,2)$&$(2,1,5)$\\
                            & $a_2=(1)$& \\
        \hline
        $(1,1,1,1,1,1,2)$ & $a_1=(1,1,1,1,1,1,2)$& $(4,4)$\\
        \hline
        $(1,1,2,2,1,1)$ & $a_1=(1,1,2)$& $(2,1,2,1,2)$\\
                        & $a_2=(2)$& \\
                        &  $a_3=(1,1)$& \\
        \hline
        $(1,1,2,1,2,1)$ & $a_1=(1,1,2,1,2)$ &$(5,1,2)$\\
                          & $a_2=(1)$& \\
        \hline
        $(1,1,2,1,1,2)$ & $a_1=(1,1,2)$&$(2,2,2,2)$ \\
                          & $a_2=(1,1,2)$& \\
        \hline
        $(1,1,1,2,2,1)$ & $a_1=(1,1,1,2)$& $(1,1,1,4)$\\
                          & $a_2=(2)$ &\\
                          & $a_3=(1)$ &\\
        \hline
        $(1,1,1,2,1,2)$ & $a_1=(1,1,1,2,1,2)$ &$(1,7)$\\
        \hline            
        $(1,1,1,1,2,2)$ & $a_1=(1,1,1,1,2)$ &$(3,1,1,3)$\\
                          & $a_2=(2)$& \\
        \hline
        $(1,1,2,2,2)$   & $a_1=(1,1,2)$&$(2,1,1,1,1,2)$ \\
                          & $a_2=(2)$ &\\
                          & $a_3=(2)$& \\
        \hline 
        $(2,1,1,1,1,1,1)$ & $a_1=(2)$& $(1,6,1)$\\
                            & $a_2=(1,1,1,1,1,1)$ &\\
        \hline
        $(2,2,1,1,1,1)$   & $a_1=(2)$& $(1,1,4,1,1)$\\
                            & $a_2=(2)$& \\
                            & $a_3=(1,1,1,1)$ &\\
        \hline 
        $(2,1,2,1,1,1)$   & $a_1=(2,1,2)$& $(4,3,1)$\\
                            & $a_2=(1,1,1)$& \\
        \hline
        $(2,1,1,2,1,1)$   & $a_1=(2)$&$(1,2,2,2,1)$ \\
                            & $a_2=(1,1,2)$ &\\
                            & $a_3=(1,1)$& \\
        \hline 
        $(2,1,1,1,2,1)$   & $a_1=(2)$& $(1,1,1,4,1)$\\
                            & $a_2=(1,1,1,2)$ &\\
                            & $a_3=(1)$ &\\
        \hline
        $(2,1,1,1,1,2)$   & $a_1=(2)$ &$(1,3,3,1)$\\
                            & $a_2=(1,1,1,1,2)$ &\\
        \hline
        $(2,2,2,1,1)$     & $a_1=(2)$&$(1,1,1,2,1,1,1)$ \\
                            & $a_2=(2)$& \\
                            & $a_3=(2)$ &\\
                            & $a_4=(1,1)$ &\\
        \hline
        $(2,2,1,2,1)$     & $a_1=(2)$ &$(1,4,1,1,1)$\\
                            & $a_2=(2,1,2)$ &\\
                            & $a_3=(1)$& \\
        \hline 
        $(2,2,1,1,2)$     & $a_1=(2)$&$(1,1,2,2,1,1)$ \\
                            & $a_2=(2)$& \\
                            & $a_3=(1,1,2)$&\\
        \hline 
        $(2,1,2,2,1)$     & $a_1=(2,1,2)$& $(4,1,1,1,1)$\\
                            & $a_2=(2)$ &\\
                            & $a_3=(1)$& \\
        \hline
        $(2,1,2,1,2)$     & $a_1=(2,1,2,1,2)$ &$(7,1)$\\
        \hline
        $(2,1,1,2,2)$     & $a_1=(2)$ &$(1,2,1,1,2,1)$\\
                            & $a_2=(1,1,2)$ &\\
                            & $a_3=(2)$ &\\
        \hline
        $(2,2,2,2)$       & $a_1=(2)$& $(1,1,1,1,1,1,1,1)$\\
                            & $a_2=(2)$ &\\
                            & $a_3=(2)$ &\\
                            & $a_4=(2)$& \\
    \end{tabular}
    \caption{The decomposition of each composition $a\in A_8$, and the resulting composition $b\in B_8$.}
    \label{decomp}
\end{table} }

\end{document}